\documentclass[a4paper,english]{amsart}

\usepackage[utf8]{inputenc}
\usepackage[T1]{fontenc}

\usepackage{amsfonts}	
\usepackage{amsmath}	
\usepackage{amssymb}	
\usepackage{amsthm}	
\usepackage{bbm}	
\usepackage{mathrsfs}
\usepackage{mathtools}	
\usepackage{stmaryrd}
\usepackage{bigints}
\usepackage{comment}

\usepackage[pdftex, pdfstartview=FitW,colorlinks=true,linkcolor=blue,%
urlcolor=blue,citecolor=blue,
pdfauthor={Timothée Bénard and Weikun He},
pdftitle={Effective Brascamp-Lieb inequalities},
pdfkeywords={Brascamp-Lieb inequalities}
]{hyperref}
\usepackage{cleveref}	
\usepackage{url}	
\usepackage{enumitem}	
\setenumerate[0]{label=\alph*)}   

\newtheorem{thm}{Theorem}[section]

\newtheorem{corollary}[thm]{Corollary}
\newtheorem{lemma}[thm]{Lemma}

\newtheorem*{thm*}{Theorem}
\newtheorem*{corollary*}{Corollary}

\theoremstyle{definition}
\newtheorem*{definition}{Definition}
\newtheorem*{example}{Example}

\newenvironment{remark}{\par\medskip \noindent \textit{Remark.} \rmfamily}{\medskip}

\newcommand{\comW}[1]{\marginpar{{\color{blue} #1}}}

\newcommand{\N}{\mathbb{N}}
\newcommand{\Z}{\mathbb{Z}}

\newcommand{\R}{\mathbb{R}}

\newcommand{\sN}{\mathscr{N}}

\newcommand{\eps}{\varepsilon}

\newcommand{\bba}{\boldsymbol{\alpha}}

\newcommand{\1}{\mathbbm{1}}

\newcommand{\bbR}{{\bf{R}}}

\renewcommand{\phi}{\varphi}
\newcommand{\cN} {{\mathcal N}}

\newcommand{\sD} {{\mathscr D}}

\newcommand{\sE}{\mathscr{E}}

\newcommand{\sA}{\mathscr{A}}

\newcommand{\bR}{\mathbf{R}}

\DeclareMathOperator{\End}{End}

\DeclareMathOperator{\GL}{GL}
\DeclareMathOperator{\id}{id}

\DeclareMathOperator{\Ell}{\mathscr{L}}

\DeclareMathOperator{\Span}{Span}

\DeclareMathOperator{\dist}{d}

\DeclareMathOperator{\rk}{rk}

\renewcommand{\parallel}{\mathbin{/\mkern-5mu/}}

\newcommand{\abs}[1]{\lvert#1\rvert}    

\newcommand{\norm}[1]{\lVert#1\rVert}   

\newcommand{\normbig}[1]{\bigl\lVert#1\bigr\rVert}

\newcommand{\set}[1]{\{\, #1  \,\}}     

\newcommand{\setBig}[1]{\Bigl\{\, #1 \,\Bigr\}}
\newcommand{\llrr}[1]{\llbracket#1\rrbracket}       

\DeclareMathOperator{\Id}{Id}

\DeclareMathOperator{\Gr}{Gr}
\DeclareMathOperator{\Ker}{Ker}
\DeclareMathOperator{\BL}{BL}

\renewcommand{\subset}{\subseteq}

\author{Timoth\'ee B\'enard }
\address{CNRS – LAGA, Universit\'e Sorbonne Paris Nord, 99 avenue J.-B. Cl\'ement, 93430 Villetaneuse}
\email{benard@math.univ-paris13.fr}

\author{Weikun He}
\address{State Key Laboratory of Mathematical Sciences, Academy of Mathematics and System Science, Chines Academy of Sciences, Beijing 100190, China}
\email{heweikun@amss.ac.cn}
\thanks{W.H. is supported by the National Key R\&D Program of China (No. 2022YFA1007500) and the National Natural Science Foundation of China (No. 12288201).}


\title{Effective Brascamp-Lieb inequalities}

\begin{document}
\begin{abstract}
We establish an effective upper bound for the Brascamp-Lieb constant associated to a weighted family of linear maps.
\end{abstract}
\maketitle
\large

\setcounter{tocdepth}{2}    
\tableofcontents

\section{Introduction} \label{Sec-intro}
A \emph{Brascamp-Lieb datum} is a tuple
$$\sD=(H, (H_{j})_{j\in J}, (\ell_j)_{j \in J}, (q_{j})_{j \in J}) $$
where $J$ is a finite index set, $H$ and $H_{j}$ are finite dimensional real Hilbert spaces (a.k.a. Euclidean spaces), $\ell_{j} : H \to H_{j}$ are
linear maps from $H$ to $H_{j}$, and $q_{j}>0$ are positive real numbers. Below, we simply write
$$\sD=((\ell_j)_{j \in J}, (q_{j})_{j \in J})$$
and keep implicit the notation $H,H_{j}$.

The \emph{Brascamp-Lieb constant of $\sD$} is the smallest constant $\BL(\sD)\in [0, +\infty]$ such that
\begin{equation}\label{def-BL-cst}
\int_{H} \prod_{j \in J} (f_{j}\circ \ell_{j})^{q_{j}} \leq \BL(\sD) \prod_{j \in J} \left(\int_{H_{j}} f_{j} \right)^{q_{j}},
\end{equation}
for all collections $(f_{j})_{j \in J}$ of measurable functions $f_j  \colon H_{j} \rightarrow \R_{\geq 0}$.

The Brascamp-Lieb inequality \eqref{def-BL-cst} unifies and generalises a number of classical results (corresponding to specific data $\sD$) such as the Cauchy-Schwarz inequality on $L^2(H)$, Hölder's inequality, Young's inequality,  the Loomis-Whitney inequality.
The study of \eqref{def-BL-cst} was initiated by Brascamp and Lieb in \cite{BL76}, and has since  been the subject of extensive research (see \cite{BCCT08} and \cite{Zhang_survey} for a historical perspective).

One of the cornerstones in the study of Brascamp-Lieb inequalities is the finiteness criterion, due to Barthe \cite{Barthe98} in rank one  (i.e. $\dim H_j=1$ for all $j$), see also \cite{CLL04},
and due to Bennett-Carbery-Christ-Tao in arbitrary rank \cite{BCCT08, BCCT10}.
It asserts that $\BL(\sD)$ is finite if and only if the following two conditions hold.
\begin{enumerate}
\item \emph{Global criticality (scaling condition)}:
\begin{equation}\label{eq:gc}
\sum_{j \in J} q_j \dim H_j = \dim H.
\end{equation}
\item \emph{Algebraic perceptivity}\footnote{We justify this terminology at the very end of \Cref{Sec-intro}.}: for every subspace $W \subseteq H$,
\begin{equation}\label{eq:sumrk}
\sum_{j \in J} q_j \dim \ell_j(W) \geq \dim W.
\end{equation}
\end{enumerate}

With this criterion in mind, a natural question arises : when $\BL(\sD)$ is finite, is there a way to determine how large or how small it is?

To the best of our knowledge, explicit values of $\BL(\sD)$ are known in the following sporadic cases:
the Hölder inequality, Young's inequality~\cite{Beckner} and its generalisation\footnote{Brascamp-Lieb~\cite{BL76} considered rank one data and reduced the determination of $\BL(\sD)$ to finding positive solutions to a system of polynomial equations.} in \cite{BL76}, the Loomis-Whitney inequality~\cite{LoomisWhitney} and its generalisation in \cite{Finner}, the geometric Brascamp-Lieb inequality~\cite{Ball, Barthe98} (in the sense of \cite[Definition 2.1]{BCCT08}).

Moreover, a large class (see \cite[Theorem 7.13]{BCCT08}) of Brascamp-Lieb data are equivalent (in the sense of \cite[Definition 3.1]{BCCT08}) to a geometric datum.
If the intertwining transformations can be explicitly determined, then the corresponding Brascamp-Lieb constant is then also explicit (by \cite[Lemma 3.3]{BCCT08}).
This can be utilized, for example, to determine the optimal constant in the affine-invariant Loomis-Whitney inequality.



These observations motivate an approach to estimating general Brascamp–Lieb constants by looking for equivalences with geometric data. This strategy is pursued by Garg–Gurvits–Oliveira–Wigderson~\cite{GGOW}, who develop a time-efficient algorithm for computing a Brascamp–Lieb constant to arbitrary precision under the additional assumption that the datum $\sD$ is rational. In this setting, they derive an upper bound for $\BL(\sD)$ in terms of the rational complexity of $\sD$ \cite[Theorem 1.4]{GGOW}. For further discussion of computational and algorithmic aspects of Brascamp–Lieb inequalities, we refer the reader to~\cite{GGOW}.

A different approach to estimate Brascamp-Lieb constants is given by Gressman~\cite[Lemma 2]{Gressman}.
This approach, using geometric invariant theory, gives for rational weights $(q_j)_{j \in J}$, the existence of rational functions in the variables $(\ell_j)_{j \in J}$ that control $\BL((q_j)_{j \in J}, (\ell_j)_{j \in J})$. These rational functions depends on $(q_j)_{j \in J}$ and are not explicit.

In this paper, we address the problem of estimating Brascamp-Lieb constants from another point of view. We establish effective estimates with a conceptual geometric interpretation. In particular, no rationality assumption is required.

\bigskip
\noindent{\bf Applications and motivation.}
Our geometric perspective has  applications to Kakeya-type problems, see the discussion in \S\ref{Sec-appli}.
We also derive a \emph{visual inequality} for covering numbers, \Cref{thm-visual-ineq}. This inequality is the original motivation for this work. In the companion paper \cite{BH25-walks}, we use it to derive a discretized subcritical projection theorem under optimal geometric assumptions, and ultimately establish effective equidistribution results for random walks on homogeneous spaces. This application aims to make the celebrated work of Benoist-Quint \cite{BQ1} effective.

The strategy of using Brascamp-Lieb inequalities in order to derive a subcritical projection theorem is inspired by Gan \cite{Gan24}.
Therein, Gan considers the continuous (non-discretized) setting and exploits  partially explicit Brascamp-Lieb estimates due to Maldague \cite{Maldague22} to give an upper bound on the Hausdorff dimension of the exceptional set, see \cite[Theorem 1]{Gan24}.
Maldague's work  \cite{Maldague22} is not enough to deduce a discretized projection theorem because of an uncontrolled constant in the upper bound in \cite[Theorem 1.1]{Maldague22}.
Theorems \ref{thm:ub}, \ref{pr:rlub}  below provide a fully explicit upper bound for Brascamp-Lieb constants.
In \cite{BH25-walks}, they will not only provide the tools required to generalize \cite[Theorem 1]{Gan24} to the discretized setting, but also allow for a  more straightforward proof of that result (for example, no need to rely on a multilinear Kakeya upgrade of the Brascamp-Lieb inequalities).

\bigskip

\subsection{Bounds for Brascamp-Lieb constants} \label{Sec-bnd-BL1}
Let  $\sD=((\ell_{j})_{j}, (q_{j})_{j})$ be a Brascamp-Lieb datum.
We  present explicit upper and lower bounds for $\BL(\sD)$ that  refine several known qualitative results, such as the finiteness criterion of Bennett-Carbery-Christ-Tao \cite{BCCT08, BCCT10}, and the local-boundedness result of  Bennett-Bez-Flock-Lee \cite{BBFL18}.

\bigskip

We start by introducing the quantities that play a role in the bounds below.
First, we refine the notion of  rank for  a linear map between Euclidean spaces, taking the metric into account.
\begin{definition}[Essential rank]\label{def-ess-rank}
Let $H,H'$ be Euclidean spaces, let $\ell:H\rightarrow H'$ be a linear map, let $\alpha \geq 0$.
The  \emph{$\alpha$-essential rank} of $\ell$ is the number (counted with multiplicity) of  singular values of $\ell$ that are strictly greater than $\alpha$. We denote it by $\rk_\alpha(\ell)$.
\end{definition}

Clearly, $\rk_0(\ell) = \rk(\ell)$ is nothing more than the rank of $\ell$.
Several characterisations of $\rk_{\alpha}$ will be given in \Cref{Sec-carac-rank-percep}.
For a given subspace $W \subseteq H$, we write
\[
\rk_\alpha(\ell \mid W) = \rk_\alpha( \ell_{\mid W})
\]
for the $\alpha$-essential rank of $\ell$ restricted to $W$.

\begin{definition}[Essential acuity] \label{def-acuity}
Let $\bba=(\alpha_{j})_{j\in J} \in \R_{\geq 0}^J$, and $W\subseteq H$ a subspace. The \emph{$\bba$-essential acuity of $\sD$ within $W$} is
$$\sA_{\bba}(\sD \mid W) = \sum_{j\in J} q_{j} \rk_{\alpha_{j}}(\ell_{j} \mid W).$$
\end{definition}

The algebraic perceptivity condition \eqref{eq:sumrk} amounts to $\sA_{0}(\sD \mid W)  \geq \dim W$ for all $W$. The next definition can be seen as a quantified  version of \eqref{eq:sumrk} involving the metric.

\begin{definition}[Metric perceptivity]
We say $\sD$ is \emph{$\bba$-perceptive}  if for every subspace $W\subseteq H$, we have
\begin{equation}\label{def-alpha-percep}
\sA_{\bba}(\sD \mid W)  \geq \dim W.
\end{equation}
See \S\ref{Sec-percep} for alternative characterisations in the context where $\ell_{j}$ are orthogonal projectors.
\end{definition}

We will also make use of the constant
\begin{equation}\label{eq:sE} \sE(\sD)=\prod_{j\in J}q_{j}^{-q_{j}\dim H_{j}/2}.
\end{equation}
$ \sE(\sD)$ can be interpreted as some weighted \emph{exponential entropy} for the vector $(q_{j})_{j\in J}$ seen as a measure on $J$.
It is dominated by \[
\sE(\sD) \leq \exp\left( {\textstyle\frac{1}{2 e} \sum_{j \in J} \dim H_j}\right).
\]

\begin{thm}[Upper bound] \label{thm:ub}
Let $\sD = \bigl((\ell_j)_{j \in J}, (q_j)_{j \in J}\bigr)$ be a Brascamp-Lieb datum. Assume $\sD$ globally critical (i.e. \eqref{eq:gc} holds) and $\bba$-perceptive for some $\bba\in \R_{> 0}^J$. Then writing $d=\dim H$, we have
\[
\BL(\sD) \leq d^{\frac{d}{2}} \sE(\sD) \prod_{j \in J} \alpha_j^{- q_j \dim H_j}.
\]
\end{thm}

To complete this result, we record a lower bound. Below, we write $\|\ell_{j}\|$ for the operator norm of $\ell_{j}$.
\begin{thm}[Lower bound] \label{thm:lb}
Let $\sD = \bigl((\ell_j)_{j \in J}, (q_j)_{j \in J}\bigr)$ be a Brascamp-Lieb datum.  Set $C:=1+\sup_{j\in J} \|\ell_{j}\|$.
Then for all $\alpha \in (0, 1]$ and all $W \in \Gr(H)$,
\[
\BL(\sD) \geq \left(C^2\sum\nolimits_{j \in J} q_j \right)^{-\dim H / 2}  \alpha^{\sum_{j \in J} q_j \rk_{\alpha}(\ell_j \mid W) - \dim W}.
\]
\end{thm}
\bigskip

As an illustration, let us compute  those bounds in the context of Young's inequality.
\begin{example}
(Young's inequality)
Consider the Brascamp-Lieb datum $\sD = \bigl( (\ell_1, \ell_2, \ell_3), (q_1, q_2, q_3) \bigr)$ with $H = \R^2$, $H_1 = H_2 = H_3 = \R$,  $\ell_1(x,y) = x$, $\ell_2(x,y) = y$, $\ell_3(x,y) = x - y$.

Then the condition of \Cref{thm:ub} is satisfied for a vector $\bba=(\alpha_1, \alpha_2, \alpha_3) \in \R_{> 0}^3$ if
$q_1 + q_2 + q_3 = 2$,  $0 < q_1, q_2, q_3 \leq 1$, and
\[
\left\{
\begin{array}{l}
\alpha_1^2 + (\alpha_1 + \alpha_3)^2 < 1 \\
\alpha_2^2 + (\alpha_2 + \alpha_3)^2 < 1.
\end{array}
\right.
\]
In particular, it is satisfied for all $\alpha_1 = \alpha_2 = \alpha_3 < \frac{1}{\sqrt{5}}$.
\Cref{thm:ub} gives $\BL(\sD) \leq \frac{10}{q_1^{q_1/2}q_2^{q_2/2}q_3^{q_3/2}}$.
On the other hand, \Cref{thm:lb} (applied with $\alpha=1$) gives $\BL(\sD) \geq \frac{1}{18}$.
We can compare these estimates to the actual value of $\BL(\sD)$, see \cite[Example 1.5]{BCCT08}, namely $\BL(\sD)= \left(\prod_{j=1}^3 \frac{(1 - q_j)^{1 - q_j}}{q_j^{q_j}}\right)^{1/2}$.
\end{example}
\bigskip

\noindent {\bf A useful variant.}
Before using the above theorems, it can be relevant to apply a change of variables in order to reduce to the case where the $\ell_{j}$ are orthogonal projectors. We present this reduction. Let $\sD = \bigl((\ell_j)_{j \in J}, (q_j)_{j \in J}\bigr)$ be a Brascamp-Lieb datum.  For every $j$, assume $\ell_{j}(H)=H_{j}$. Write $\pi_{j}:=\pi_{\parallel \Ker \ell_{j}} :H \rightarrow H$ the orthogonal projector with same kernel as $\ell_{j}$, set $\sD_{\mathrm{proj}}:=((\pi_{j})_{j\in J}, (q_{j})_{j\in J})$ the associated datum. Then by\footnote{This follows indeed from \cite[Lemma 3.3]{BCCT08} by observing that $\sD$ and $\sD_{\mathrm{proj}}$ are equivalent in the sense of \cite[Definition 3.1]{BCCT08}, due to the equality $\ell_{j}=\ell_{j | (\Ker \ell_{j})^\perp}\circ  \pi_{j}$.} \cite[Lemma 3.3]{BCCT08}, we have
\begin{equation}\label{var-change}
\BL(\sD)= \Upsilon(\sD)\BL(\sD_{\mathrm{proj}}) \,\,\,\text{ where }\,\,\,
\Upsilon(\sD) = \prod_{j \in J} \abs{\det\nolimits_{\ker(\ell_j)^\perp \to H_j}(\ell_j)}^{- q_j},
\end{equation}
and $\det\nolimits_{\ker(\ell_j)^\perp \to H_j}(\ell_j)$ denotes the determinant of the linear isomorphism $(\ell_{j})_{|\ker(\ell_j)^\perp} : \ker(\ell_j)^\perp \to H_j$ (with respect to orthonormal bases).

Combining    \eqref{var-change} with \Cref{thm:ub}, we get that if $\sD_{\mathrm{proj}}$ is globally critical and $\bba$-perceptive, then
\begin{equation}\label{up-v2}
\BL(\sD) < d^{\frac{d}{2}} \sE(\sD) \Upsilon(\sD) \prod_{j \in J} \alpha_j^{- q_j \dim H_j}.
\end{equation}
Similarly, a variant of \Cref{thm:lb} can be deduced from \eqref{var-change}.

Although slightly more complicated to formulate, this version of the upper bound has the advantage that the perceptivity hypothesis only concerns the geometric datum $\Ker \ell_{j}$ and the weights $q_{j}$, while the way the linear maps $\ell_{j}$ may distort space is fully captured (without loss) by the parameter $\Upsilon(\sD)$.  Another benefit is that perceptivity has additional characterisations in the context of orthogonal projectors. We will see this in \Cref{Sec-carac-rank-percep}.

An example where \eqref{up-v2} is better than \Cref{thm:ub} is given by $\sD=\sD_{\lambda} := \bigl((\ell_1, \ell_2, \ell_{3}), (\frac{1}{2},\frac{1}{2}, \frac{1}{2})\bigr)$ where $\ell_{i}:\R^3\rightarrow \R^2$, and $\ell_{1}(x)=(x_{2}, x_{3})$, $\ell_{2}(x)=(x_{1}, x_{3})$, $\ell_{3}(x)=(x_{1}, \lambda x_{2})$, with  $\lambda>0$ fixed and small.  For more details, see \Cref{Sec-appendix}.

\bigskip
\noindent{\bf Qualitative consequences.}
From Theorems \ref{thm:ub}, \ref{thm:lb}, we  derive a few qualitative corollaries.

\bigskip
\noindent{\bf \normalsize Finiteness.}
Our approach consists in making effective the proof of the aforementioned finiteness criterion from \cite{BCCT08}. Therefore, it is not surprising that Theorems \ref{thm:ub}, \ref{thm:lb} imply the latter.

To recover the criterion, note  on the one hand that if \eqref{eq:gc} fails, then by scaling consideration, $\BL(\sD) = +\infty$.
If \eqref{eq:sumrk} fails for some subspace $W \subseteq H$, then we have for all $\alpha \in (0,1)$,
\[
\sA_{\alpha}(\sD\mid W)  - \dim W \leq \sum_{j \in J} q_j  \dim\ell_{j}(W) - \dim W < 0.
\]
Letting $\alpha \to 0$ in \Cref{thm:lb}, we obtain $\BL(\sD) = +\infty$.

On the other hand, assume $\sD$ satisfies \eqref{eq:gc}  and  \eqref{eq:sumrk}. Note that given $W$, we have $\sA_{\bba}(\sD|W')= \sA_{0}(\sD|W')$ for $\bba\in \R_{>0}^J$ close enough to zero and $W'$ close enough to $W$. By compactness of the Grassmanian $\Gr(H)$, we deduce that $\sD$ is $\bba$-perceptive for $\bba\in \R_{>0}^J$ in a neighorhood of $0$, whence $\BL(\sD) < + \infty$ by \Cref{thm:ub}.

\bigskip
\noindent{\bf \normalsize Local boundedness.} 
The condition that $\sD$ is  $\bba$-perceptive is stable under small perturbation of the $\ell_{j}$.
Thus, \Cref{thm:ub} implies the following stability result, originally due to Bennett-Bez-Flock-Lee~\cite{BBFL18}: if $\BL(\sD) < +\infty$ for some Brascamp-Lieb datum $\sD = (\boldsymbol{\ell}, \mathbf{q})$, then fixing $\mathbf{q}$, there is a constant $C > 0$ such that $\BL(\boldsymbol{\ell}', \mathbf{q}) < C$ holds uniformly for $\boldsymbol{\ell}'$ ranging in a small neighborhood of $\boldsymbol{\ell} \in \prod_{j \in J} \Ell(H, H_j)$.

Here $\Ell(H, H_j)$ denotes the set of linear maps from $H$ to $H_{j}$.
Let us mention that continuity and differentiability of $\BL(\boldsymbol{\ell}, \mathbf{q})$ in the variable $\boldsymbol{\ell}$ are investigated by Valdimarsson~\cite{Valdimarsson}, Bennett-Bez-Cowling-Flock~\cite{BBCF} and Garg-Gurvits-Oliveira-Wigderson~\cite{GGOW}.

\bigskip
\noindent{\bf  \normalsize Joint local boundedness for simple data.}
The following joint local boundedness seems to be new.
Recall that a globally critical Brascamp-Lieb datum is said to be \emph{simple} if \eqref{eq:sumrk} holds as a strict inequality for any proper nonzero subspace $W \in \Gr(H)$.
\begin{corollary}[Joint local boundedness in $(\ell_j)_j$ and $(q_j)_j$ near simple data]
Let $\sD$ be a simple Brascamp-Lieb datum.
Then there exists a constant $C > 0$ and a neighborhood $U$ of $\sD$ in $\prod_{j \in J} \Ell(H,H_j) \times \R_{>0}^J$ such that for any $\sD' \in U$ that is globally critical, we have $\BL(\sD') < C$.
\end{corollary}

\bigskip
\subsection{Bounds for   localised regularised Brascamp-Lieb constants.} \label{Sec-bndBL2}
We present effective  \emph{regularised} and \emph{localised}  Brascamp-Lieb inequalities.
They imply Theorems \ref{thm:ub}, \ref{thm:lb}, and have the advantage to be meaningful for  Brascamp-Lieb data that do not necessarily satisfy \eqref{eq:gc} or \eqref{eq:sumrk}.
Those estimates strengthen former quantitative bounds due to Maldague \cite{Maldague22}.

\bigskip
We first introduce the corresponding definitions.
Given a Euclidean space $H$, a positive definite symmetric endomorphism $R \in \End(H)$, and $x\in H$, set
\begin{equation}\label{def-chi-sN}
\chi_{R}(x)= e^{- \pi \langle x, R x \rangle } \quad \text{ and }\quad \sN_{R}(x)=(\det R)^{1/2}\chi_{R}(x).
\end{equation}
We interprete $\chi_{R}$ as a Gaussian truncation function, while $\sN_{R}$ is the centered normal probability density of covariance $(2\pi R)^{-1}$.
A function $f : H \to \R_{\geq 0}$ is said to be of \emph{type $R$} if it takes the form of a convolution $f = \sN_R * \mu$ for some finite Borel measure $\mu$ on $H$.
In essence, being of type $R$ expresses that a function comes from the mollification of a positive measure at a scale given by $R$ (say $\sN_{R}^{-1}[1, +\infty)$). As $R$ gets bigger, the scale gets smaller, so the condition becomes less restrictive.

A   localised regularised Brascamp-Lieb datum is a triple $(\sD, \bbR, T)$ where
$\sD = ((\ell_j)_{j \in J}, (q_j)_{j \in J})$ is a Brascamp-Lieb datum,
$\bbR = (R_{j})_{j \in J}$ is a collection such that each $R_j$ is a positive definite symmetric endomorphism of $H_j$
and $T$ is a positive definite symmetric endomorphism of $H$.

The  localised regularised Brascamp-Lieb constant of $(\sD, \bbR, T)$ is the smallest number $\BL(\sD, \bbR, T)\in [0,+\infty]$ such that
\[
\int_{H} \chi_T \prod_{j \in J} (f_{j}\circ \ell_{j})^{q_{j}} \leq \BL(\sD, \bbR, T) \prod_{j \in J} \left(\int_{H_{j}} f_{j} \right)^{q_{j}},
\]
for all collections $(f_{j})_{j \in J}$ such that $f_j :  H_{j}\rightarrow \R_{\geq0}$ is of type $R_j$ for each $j \in J$. In sum,  \emph{the parameter $\bbR$ imposes a certain class of regular functions, while the parameter $T$ yields a gradual truncation along the integral.}
This definition (with $T = \Id_H$) is extracted from  \cite[Section 8]{BCCT08}.

\begin{remark}
A variant of $\BL(\sD, \bbR, T)$, asking for a straight regularisation (inputs $f_{j}$ tile-wise constant) and a straight truncation ($\chi_T$ replaced by $\1_{B_1^H}$) is used in \cite{BCCT10, Maldague22, Zorin}. Upper bound estimates for this variant variant follow from counterparts on  $\BL(\sD, \bbR, T)$ using suitable mollification, as in the proof of \Cref{thm-visual-ineq} below.
\end{remark}

\bigskip
We give an explicit upper bound on $\BL(\sD, \bbR, T)$.
In situations  where \eqref{eq:gc} is not satisfied, we use the \emph{total acuity}
$$\sA(\sD)=\sum_{j\in J}q_{j}\dim H_{j}.$$

For situations where \eqref{eq:sumrk} is violated, we need the following notion.
\begin{definition}[Metric perceptivity 2]
Given $\bba\in \R_{\geq0}^J$, $\beta\in \R_{\geq0}$, we say $\sD$ is \emph{$(\bba, \beta)$-perceptive}  if for every subspace $W\subseteq H$, we have
\begin{equation}\label{def-alpha-beta-percep}
\sA_{\bba}(\sD \mid W)  \geq \dim W -\beta.
\end{equation}
\end{definition}

We also use the norm
\[
N(\sD, \bbR, T)=\big\|T+ \sum_{j \in J} q_{j} \ell^*_{j} \,R_{j}\, \ell_{j} \big\|
\]
where $\ell^*_{j}:H_{j}\rightarrow H$ stands for the adjoint of $\ell_{j}$. Finally, we recall the quantity $\sE(\sD)$ has been defined in \eqref{eq:sE}.

\begin{thm}[Upper bound]\label{pr:rlub}
Let $(\sD,\bbR, T)$ be a  localised regularised Brascamp-Lieb datum, let  $\bba \in \R_{>0}^J$ and $\beta \in \R_{\geq 0}$. Assume $\sD$ is $(\bba, \beta)$-perceptive and  $\rk_{\alpha_j}(\ell_j) = \dim H_j$ for each $j \in J$.
Then, writing $d = \dim H$, we have
\begin{equation}\label{eq:rlub}
\BL(\sD, \bbR, T)
\leq d^{\frac{\sA(\sD)}{2}}  \sE(\sD) \prod_{j \in J} \alpha_j^{-q_j \dim H_j}  N(\sD, \bbR, T)^{\frac{\sA(\sD)-d+\beta}{2}} \norm{T^{-1}}^{\frac{\beta}{2}}.
\end{equation}
\end{thm}

\begin{remark}
The condition $\rk_{\alpha_j}(\ell_j) = \dim H_j$ can be interpreted as a quantitative strengthening of the condition that $\ell_j \colon H \to H_j$ is surjective.
The latter is often required in the form of a non-degeneracy condition.
\end{remark}

\begin{remark} In the context of \Cref{pr:rlub}, we have $\sA(\sD)=d$ and $\beta=0$, so $\bbR$ and $T$ play no role in the above upper bound.
The upper bound from \Cref{thm:ub} can then be seen as a limit case of \Cref{pr:rlub}, by virtue of the heuristic $\BL(\sD)=\BL(\sD, \infty, 0)$ (see \Cref{Sec-compute-BL} for a detailed proof).

In fact, upper bounds on other variants of the Brascamp-Lieb inequalities can be deduced similarly. For instance, assuming the relation $\sA(\sD)-d+\beta=0$, we see $\bbR$ plays no role in the above upper bound, whence we may pass to the limit to bound in the same manner the \emph{localised} Brascamp-Lieb constant $\BL(\sD, \infty, T)$.
Note the appearance of the condition $\sA(\sD)-d+\beta=0$ is to be expected.
Indeed, set $\beta_{\min}=\beta_{\min}(\sD)$ to be the smallest $\beta\geq0$ such that $\sD$ is $(\bba, \beta)$-perceptive for some $\bba \in [0,1)^J$.
In other words,
\[
\beta_{\min} := \sup_{W \in \Gr(H), \bba \in [0,1)^J} \dim W - \sA_{\bba}(\sD \mid W) = \sup_{W \in \Gr(H)} \dim W - \sA_{\mathbf{0}}(\sD \mid W).
\]
Then it follows from \cite[Theorem 2.2]{BCCT10} that the relation $\sA(\sD)-d+\beta_{\min}=0$
is necessary and sufficient for the finiteness of $\BL(\sD, \infty, T)$.

Another limiting upper bound is that of $\BL(\sD, \bbR, 0)$ under the assumption $\beta=0$, also a necessary and sufficient condition for finiteness of $\BL(\sD, \bbR, 0)$ when $\beta=\beta_{\min}$, see \cite{BCCT10, Maldague22}.
\end{remark}

\begin{remark}
 The exponents $\frac{\sA(\sD)-d+\beta}{2}$ and $\frac{\beta}{2}$ appearing in  \Cref{pr:rlub} are \emph{optimal}. This optimality is justified at the end of \Cref{Sec-visual-inequality} through the connection between \Cref{pr:rlub} and the visual inequality established therein. For the second exponent, optimality is also a direct consequence   of \cite[Theorem 1 (lower bound)]{Maldague22}.
\end{remark}

We also record the following lower bound for localised regularised Brascamp-Lieb constants. For simplicity, we  assume  $R_j = \id_{H_j}$ for each $j \in J$ and $T$ is a contracting homothety.
\begin{thm}[Lower bound]\label{thm:loc-lb}
Let $\sD = \bigl( (\ell_j)_j, (q_j)_j \bigr)$ be a Brascamp-Lieb datum. Set $C:=1+\sup_{j\in J} \|\ell_{j}\|$.
Let $\bR = (\id_{H_j})_{j \in J}$ and $T = t \id_H$ for some $t>0$.

Then for all $\alpha \in (0, 1]$ and all $W \in \Gr(H)$, we have
\[
\BL( \sD, \bR, T) \geq \left( (C/\alpha)^2 t+ C^2\sum\nolimits_{j \in J} q_j \right)^{-\dim H / 2} \alpha^{\sum_{j \in J} q_j \rk_{\alpha}(\ell_j \mid W) - \dim W}.
\]
\end{thm}

\begin{remark}
Theorems \ref{pr:rlub}, \ref{thm:loc-lb} strengthen former quantitative bounds due to Maldague \cite{Maldague22}. More precisely, \cite[Theorem 1]{Maldague22} states essentially
that for fixed $\sD$, for $\bbR=(\Id_{H_j})_{j}$ and $T = t \id_H$ with $t\in (0,1)$, we have  $\kappa t^{-d \beta_{\min} /2}\leq \BL(\sD, \bbR, T) \leq K t^{- d\beta_{\min} /2} $ where $K>\kappa>0$ are constants depending on $\sD, \bbR$. In fact the constant $\kappa$ could be made explicit from the proof in terms of  $d$, $\sup_{j}\|\ell_{j}\|$, $\sum_{j} q_{j}$.
However, the approach in \cite{Maldague22} does not allow one to track down how the constant $K$ in the upper bound depends on $\sD$ and $\bbR$. Exploiting a different strategy, our upper bound in \emph{\Cref{pr:rlub}  manages to cover the dependence on $\sD$ and $\bbR$}. This is vital to deduce the upper bound on $\BL(\sD)$ stated in  \Cref{thm:ub}.

A (weaker) locally uniform upper bound is presented in \cite[Theorem 3]{Maldague22} in the case of orthogonal projectors. This bound is weaker because it requires an exponent bigger than $\beta_{\min}/2$, thus loosing optimality. Here, \Cref{pr:rlub}  does provide a locally uniform bound while keeping the right exponent (because for fixed $(q_j)_{j \in J}$, the  condition that  $\sD = \bigl((\ell_j)_{j \in J}, (q_j)_{j \in J}\bigr)$ is $(\bba,\beta)$-perceptive is  open  in $(\ell_j)_j$.)

We also record that \Cref{thm:loc-lb}  recovers the lower bound in $\kappa t^{-d \beta_{\min} /2}$ mentioned above, by  taking $\alpha = t^{1/2}$ and using the monotonicity $\rk_\alpha(\ell_j \mid W) \leq \rk_0(\ell_j \mid W) = \dim \ell_j(W)$.
\end{remark}

\bigskip
\noindent{\bf Localised regularised variant.}
To conclude \S\ref{Sec-bndBL2}, we point out that the useful variant highlighted in \S\ref{Sec-bnd-BL1} can also be implemented in the  localised regularised setting.  Indeed,  let $(\sD,\bbR, T)$ be a  localised regularised Brascamp-Lieb datum. Provided the surjectivity condition $\ell_{j}(H)=H_{j}$ for all $j\in J$, we see as in \S\ref{Sec-bnd-BL1} that
$$\BL(\sD, \bR, T)= \Upsilon(\sD) \BL(\sD_{\mathrm{proj}}, \bR', T)  \,\,\,\,\, \,\,\,\,\, \,\,N(\sD, \bbR, T)=N(\sD_{\mathrm{proj}}, \bbR', T)$$
where $\bR':=(\varphi_{j}^*  R_{j} \varphi_{j})_{j\in J}$ with $\varphi_{j}$ the linear isomorphism given by $\varphi_{j}:=\ell_{j|(\Ker \ell_{j})^\perp} : (\Ker \ell_{j})^\perp\to H_{j}$.
Combined with \Cref{pr:rlub}, this provides an alternative upper bound for $\BL(\sD, \bR, T)$ where the perceptivity condition only concerns $\sD_{\mathrm{proj}}$.

\bigskip

\subsection{Applications} \label{Sec-appli}
Finally, we discuss some consequences of \Cref{pr:rlub}.

\bigskip
\noindent{\bf   Visual inequality.}
Let $H$ be a Euclidean space.
Given a subset $A \subseteq H$ and $\delta>0$, we denote by $\cN_\delta(A)$ the $\delta$-covering number of $A$, i.e. the minimal number of $\delta$-balls that is needed to cover $A$.

Consider a collection of lines $(L_{j})_{1 \leq j \leq d}$ in $H$ and $\alpha\in (0, 1/2)$. Write $L_{j}=\R v_{j}$ with $v_{j}$ unitary, and  $\pi_{L_{j}}$  the orthogonal projector of image $L_{j}$.
If $\norm{v_{1} \wedge \dotsm \wedge v_{d}} \geq \alpha$ and $ d=\dim H $, then it is well-known that for every $A\subseteq H$ and $\delta>0$,
\begin{equation} \label{visual-ineq0}
\cN_{\delta}(A)\ll_d \alpha^{-d} \cN_{\delta}(\pi_{L_{1}}A) \dotsm \cN_{\delta}(\pi_{L_{d}}A).
\end{equation}

As a corollary of the effective localised regularised Brascamp-Lieb estimate from \Cref{pr:rlub}, we can generalise the above inequality to arbitrary configurations of subspaces.

\begin{thm}[Visual inequality] \label{thm-visual-ineq}
Let $\sD=((\pi_{H_j})_{j\in J}, (q_{j})_{j\in J})$ be a Brascamp-Lieb datum made of orthogonal projectors $\pi_{H_j} : H \to H_j$ where $H_j \in \Gr(H)$.
Set $d = \dim H$.
Assume $\sD$ is $(\bba, \beta)$-perceptive for some $\bba = (\alpha_j)_{j} \in (0,1)^J$, $\beta\in \R_{\geq 0}$.
Then for every $\delta\in (0,1)$, every subset $A \subset B^H_{1}$, we have
\begin{equation} \label{eq-thm-visual-ineq}
\cN_{\delta}(A)\leq C \delta^{-\beta} \prod_{j \in J} \alpha_j^{-q_j \dim H_j} \prod_{j \in J} \cN_{\delta}\bigl(\pi_{H_{j}}A\bigr)^{q_{j}}
\end{equation}
where $0< C \ll_{d} O_{d}(1)^{\sum_{j}q_{j}}(1+\sum_{j \in J}q_{j})^{\frac{\sA(\sD)-d+\beta}{2}} \sE(\sD) $.
\end{thm}

In the above, $B^H_{1}$ refers to the closed centered unit ball in $H$.
\begin{remark}
The inequality \eqref{visual-ineq0} is covered by \Cref{thm-visual-ineq}, taking $J=\{1, \dotsc, d\}$ and $\ell_{j}=\pi_{L_{j}}$, $q_{j}=1$ for $1\leq j\leq d$. Indeed, writing $\alpha^{\otimes d}=(\alpha, \dots, \alpha)$ the $d$-tuple with all entries equal to $\alpha$,  we note that in this case  $\sD$ is $(O_{d}(1)\alpha^{\otimes d})$-perceptive (so $\beta=0$), and $\sA(\sD)=d$, $\sE(\sD)=1$, whence the above constant $C$ is $O_{d}(1)$.
\end{remark}

\begin{remark}
The lack of perceptivity is incarnated by $\beta$. Its role in the visual inequality can be understood as follows.
If in \eqref{visual-ineq0}, we wish to remove the first $k$ projectors ($k \geq 0$), we may only guarantee
$$\cN_{\delta}(A)\ll_{d} \alpha^{-(d-k)} \delta^{-k}\cN_{\delta}(\pi_{L_{k+1}}A) \dotsm \cN_{\delta}(\pi_{L_{d}}A)$$
as can be seen by taking $A$ of the form $A=B^{\Span(e_{1}, \dots, e_{k})}_{1}\times A'$ with $A'\subseteq \Span(e_{k+1}, \dots, e_{d-k})$.
This observation is reflected by the fact that the datum $\sD=((\pi_{L_{j}})_{k+1\leq j\leq d}, (1, \dotsc, 1) )$ is $(O_{d}(1)\alpha^{\otimes (d-k)}, k)$-perceptive.
\end{remark}

\bigskip
\noindent{\bf   Perturbed Brascamp-Lieb theorem and multilinear Kakeya inequalities.}
The upper bound of \Cref{pr:rlub} can be put into Zorin-Kranich's machine~\cite[Theorem 1.3]{Zorin} to derive effective perturbed Brascamp-Lieb inequalities and consequently multilinear Kakeya inequalities.
The effectiveness of \Cref{pr:rlub} allows us to control quantities appearing in the statements of such inequalities that were previously known to depend on the Brascamp-Lieb datum.
More precisely, previously, these quantities would depend obscurely on the Brascamp-Lieb datum, while now we know how to control them in terms of the parameters $(\bba,\beta)$ of perceptivity.
Such quantities include
\begin{enumerate}
\item the constant $\delta$ in Zhang's endpoint perturbed Brascamp-Lieb inequality~\cite[Theorem 1.11]{Zhang18} (let us mention here that this theorem generalises the celebrated multilinear Kakeya estimates of Bennett-Carbery-Tao~\cite[Theorem 1.15]{BCT06} and of Guth~\cite[Theorem 1.3]{Guth10}).
\item the constant $\nu$ and the constant $C_\epsilon$ in Maldague's generalised multilinear Kakeya inequality~\cite[Theorem 1.2]{Maldague22}.
\end{enumerate}

$$\rule{2cm}{0.4pt}$$

\bigskip

\noindent{\bf On the terminology of perceptivity.}
In the literature, condition \eqref{eq:sumrk} usually does not have a name.
We propose the terminology  of \emph{algebraic perceptivity}, and later on of \emph{metric perceptivity} for its quantified version. ``Algebraic'' refers  to linear algebra, while ``metric'' refers to the additional involvement of the Euclidean structure. The term perceptivity is motivated by the intuition that for a Brascamp-Lieb datum $\sD$ to reflect   features of $H$ without loss, we need at least $\sum q_{j} \dim \ell_{j}(H) \geq \dim H$.
Here, condition \eqref{eq:sumrk} requires that this also holds in restriction to subspaces, in other words, $\sD$ perceives all subspaces.
In \Cref{thm-visual-ineq}, the visual lexical field arises naturally: in order to perceive the size of a set, we only need to know its projection via a Brascamp-Lieb datum satisfying a perceptivity condition in the  sense defined previously.

\bigskip
\noindent{\bf Acknowledgements.}
The authors thank Jonathan Bennett,  Shukun Wu, and Ruixiang Zhang for useful comments.

\bigskip
\noindent{\bf Formalised proof.}
The central component of this work, \Cref{pr:rlub}, has been formalised in LEAN by Project Numina (\url{https://projectnumina.ai/}).
The source code is available  at \url{https://github.com/project-numina/BrascampLieb}.
We are grateful to Project Numina for their  support.

\section{Essential rank and perceptivity} \label{Sec-carac-rank-percep}

In this section, we clarify the notions of essential rank and perceptivity by presenting several properties and characterisations.

\subsection{Essential rank}

\bigskip
We let $H, H'$ denote Euclidean spaces. Recall $\Gr(H)$ denotes the Grassmanian of $H$, and $\Ell(H,H')$ is the set of linear maps from $H$ to $H'$.
We first record a straightforward continuity result for the essential rank.
\begin{lemma}
\label{lm-rank-cont}
The map
\[
\begin{array}{ccc}
\R_{\geq 0} \times \Ell(H,H') \times \Gr(H) & \to & \N,\\
(\alpha, \ell, W) & \mapsto & \rk_\alpha(\ell \mid W)
\end{array}
\]
is lower semi-continuous.
That is, for any $k \in \N$, the subset of triples $(\alpha, \ell, W)$ such that $\rk_\alpha(\ell \mid W) \geq k$ is open.
\end{lemma}

Next, we give an equivalent definition of the essential rank.
For a Euclidean space $W$ and a radius $\rho \geq 0$, we denote by $B_\rho^W$ the \emph{closed} centered ball of radius $\rho$ in $W$.
In particular, $B^W_0 = \{0\}$.
\begin{lemma}
\label{lm:arank}
For $\alpha \in \R_{\geq 0}$, $\ell \in \Ell(H,H')$,  $W \in \Gr(H)$, we have
\[
\rk_{\alpha}(\ell \mid W) = \min \setBig{\dim E \,:\, E \in \Gr(H') \text{ with } \ell \bigl(B^W_1\bigr) \subseteq B^{H'}_\alpha + E}.
\]
\end{lemma}
\begin{remark}
We see from the proof that the above minimum is realized by some $E$ of the form $E=\ell(V)$ where $V\in\Gr(W)$ satisfies $\ell_{|W}(V^\perp)\subseteq E^\perp$ and $\|\ell_{|W\cap V^\perp}\|\leq \alpha$.
\end{remark}

\begin{proof}
We need to show that $r=\widetilde r$ where
\[
r:=\rk_\alpha(\ell_{|W})  \quad \quad \widetilde{r}:=\min \setBig{\dim E \,:\, E \in \Gr(H') \text{ with } \ell \bigl(B^W_1\bigr) \subseteq B^{H'}_\alpha + E}.
\]

By the singular decomposition of $\ell_{|W} : H \to H'$, there is an orthonormal basis $(w_1, \dotsc, w_k)$ of $W$ such that the $(\ell w_i)_{1\leq i \leq k}$ are pairwise orthogonal and
\[
\sigma_1 := \norm{\ell w_1} \geq \dots \geq \sigma_k :=\norm{\ell w_k}
\]
are the singular values of $\ell_{|w}$, decreasingly ordered.
By definition of the essential rank $r$, we have $\sigma_i > \alpha$ for $1 \leq i \leq r$ while $\sigma_i \leq \alpha$ for $r + 1 \leq i \leq k$.

On the one hand, for $E_0 := \Span\{\ell w_1, \dotsc, \ell w_r \}$, we have
\(
\ell(B_1^W) \subset B^{H'}_\alpha + E_0,
\)
leading to $\widetilde r \leq \dim E_0 = r $.

On the other hand, we observe that
$$\ell(B^W_{1})\supseteq \ell(B^{\Span(w_{1}, \dots, w_{r})}_{1}) \supseteq B^{E_{0}}_{\sigma_{r}}.$$
As $\sigma_{r}>\alpha$, we obtain that $\ell(B^W_{1})$ cannot be included in  $B_\alpha^{H'} + E$ for some $E\in \Gr(H')$ with dimension $\dim E<\dim E_{0}$. Therefore $\widetilde{r}\geq r$.
\end{proof}

\bigskip

\noindent{\bf The case of orthogonal projectors.}
We present further characterisations of the essential rank in the setting of orthogonal projectors.
The next lemma provides a hands-on description in the rank-one case.

\begin{lemma}
 \label{ex-essrank-lines}
Consider a line $L =\R u\in \Gr(H)$  with $\|u\|=1$. For all $W \in \Gr(H)$,  we have
\[
\rk_\alpha(\pi_L \mid W) = \begin{cases}
1 \quad \text{if $\norm{\pi_W(u)} > \alpha$,}\\
0 \quad \text{otherwise.}
\end{cases}
\]
\end{lemma}
\begin{proof}
 Note $\rk_\alpha(\pi_L \mid W)\in \{0, 1\}$ and is equal to $1$ if and only if $\|(\pi_{L})_{|W}\|>\alpha$, i.e. there exists $w\in W$ with $\|w\|\leq 1$ such that $|\langle u, w \rangle|>\alpha$. But $\sup_{w \in B^W_{1}}|\langle u, w \rangle|= \|\pi_{W}(u)\|$, whence the result.
\end{proof}

We now turn to the case of projectors of arbitrary rank.
 Fix a Riemannian metric on $\Gr(H)$ which is invariant by the group of isometries of $H$.
We let $\dist(\cdot, \cdot)$ denote the corresponding distance on $\Gr(H)$, and for $r\geq0, W\in \Gr(H)$, we write $B_{r}(W)$ for the closed ball or radius $r$ and center $W$.
By convention, if $V, W \in \Gr(H)$ have different dimensions, then $\dist(V,W) = +\infty$.

\begin{lemma} \label{eq:comp sR}
There exists a constant $C>1$, depending only on the choice of  metric on $\Gr(H)$, such that for all $V,W \in \Gr(H)$, $\alpha \in [0,C^{-1})$,
\begin{equation}
\min_{W' \in B_{C\alpha}(W) } \dim \pi_{V}(W')
\leq
\rk_\alpha(\pi_{V} \mid W)
\leq \min_{W' \in B_{ \alpha/C}(W) } \dim \pi_{V}(W').
\end{equation}
\end{lemma}

\begin{proof}
We use the shorthand $\ell := \pi_{V}$ and $r=\rk_{\alpha}(\ell_{|W})$.

We start with the inequality on the left. Recall the basis $(w_{1}, \dots, w_{k})$ of $W$ from the proof of \Cref{lm:arank}.
Consider
\[
W'_{0} = \Span\{w_1, \dotsc, w_r, w_{r+1} - \ell w_{r+1}, \dotsc, w_k - \ell w_k \}.
\]
Using that $\ell$ is a projector, we have $ \dim \ell (W'_{0})\leq r$.
To bound $\dist(W'_{0},W)$, observe that the map from $H^{k - r}$ to $\R$ given by
\[
(v_{r+1}, \dotsc, v_k) \mapsto \dist(\Span\{w_1, \dotsc, w_r, w_{r+1} + v_{r+1}, \dotsc, w_k + v_k\}, W)
\]
is $C$-Lipschitz continuous on the ball of radius $C^{-1}$ around $0 \in H^{k-r}$ for some constant $C > 1$.
Moreover, since $\dist$ is invariant under the group of isometries and the latter acts transitively on $\Gr(H,k)$, the constant $C$ is uniform in $W \in \Gr(H,k)$.
As $\max_{i>r}\|\ell w_{i}\|\leq \alpha$ by definition,   we get  $\dist(W'_{0},W) \leq  C k\alpha$ provided $k\alpha<C^{-1}$.  This shows the  inequality on the left (with $C\dim H$ in the place of $C$).

For the  inequality on the right, observe\footnote{Use a local chart around $W'$ and then obtain uniformity using the group of isometries.} that there is a constant $C>1$ such that
for all $W, W' \in \Gr(H)$,
\[
\sup_{w \in B_1^W} \dist(w,W') < C \dist(W,W').
\]
Then for every $W' \in \Gr(H)$ such that $\dist(W,W') \leq \alpha/C$, using that $\ell$ is $1$-Lipschitz, we have
\[
\ell(B_1^W) \subset \ell(W') + B^{H}_{\alpha}.
\]
In view of \Cref{lm:arank}, this implies
\[
r\leq \dim \ell(W'),
\]
finishing the proof of the second inequality.
\end{proof}

\subsection{Perceptivity} \label{Sec-percep}


The characterisations of essential rank discussed above immediately provide further characterisations of perceptivity.

\begin{example}
As a corollary of \Cref{ex-essrank-lines},  if the tuple $(\ell_j)_{j \in J}$ in the datum $\sD$ consists of orthogonal projectors of rank one, then for $\bba\in \R_{\geq0}
^J$, the $\bba$-essential acuity within $W \in \Gr(H)$ is given by
$$\sA_{\bba}(\sD \mid W)=\sum \{q_{j}\,:\, j \in J \text{ such that }\norm{\pi_W(u_j)} > \alpha_{j}\},$$
where $u_j$ is a unit vector in the line $\ell_j(H)$.
Thus, $\alpha^{\otimes J}$-perceptivity  means in essence that for every $W \in \Gr(H)$, the sum of weights of  lines forming an angle greater than $\alpha$ with $W^\perp$ is at least $\dim W$.
This geometric criterion is simple to visualise.
\end{example}

For orthogonal projectors of arbitrary rank, the next lemma provides  useful additional insight on perceptivity. This characterisation will be exploited in \cite{BH25-walks} to derive a subcritical projection theorem from the present article.

\begin{lemma}
\label{percep-orthog-proj}
There exists a constant $C > 1$ depending only on the choice of  metric on $\Gr(H)$ such that the following holds.
Let $\sD=((\pi_{H_j})_{j\in J}, (q_{j})_{j\in J})$ be a globally critical Brascamp-Lieb datum made of orthogonal projectors $\pi_{H_j} : H \to H_j$ where $H_j \in \Gr(H)$.
Let $\bba = (\alpha_j)_{j \in J} \in [0, C^{-1})^J$ and $\beta \in \R_{\geq 0}$.

If for every $W \in \Gr(H)$,
\begin{equation*} 
\sum_{j \in J}  q_{j} \max_{W'\in B_{C\alpha_{j}}(W)} \frac{\dim H^{\perp}_{j} \cap W'}{\dim W} -\frac{\beta}{\dim W}\,\leq\, \sum_{j \in J}  q_{j} \frac{\dim H^{\perp}_{j}}{\dim H},
\end{equation*}
then $\sD$ is $(\bba,\beta)$-perceptive.

Conversely, if $\sD$ is $(\bba, \beta)$-perceptive, then the above holds with the condition $W'\in B_{C\alpha_{j}}(W)$ replaced by $W'\in B_{\alpha_{j}/C}(W)$.
\end{lemma}

Taking $\beta=0$, we see perceptivity means in essence that the proportion of kernel $H^\perp_{j}$ intersecting any given subspace $W$ is smaller (in average) than the proportion of $H^\perp_{j}$ in the whole space $H$.

\begin{proof}[Proof of \Cref{percep-orthog-proj}]
Let us show the sufficient condition.
By (the first bound in) \Cref{eq:comp sR}, we know that in order to show that $\sD$ is $(\bba,\beta)$-perceptive, it suffices to check: $ \forall W \in \Gr(H)$,
\begin{equation}\label{eq-percep-jghhd}
  \sum_{j \in J} q_j \Bigl(\dim W - \max_{W'\in B_{C\alpha_{j}}(W)} \dim (H^\perp_{j} \cap W') \Bigr) \geq \dim W - \beta.
\end{equation}
Equation \eqref{eq-percep-jghhd} can be rewritten as
\begin{equation}\label{carac-perp2.1}
\Bigl(\sum_{j \in J}  q_{j} -1\Bigr)\dim W \geq \sum_{j \in J}  q_{j}\max_{W'\in B_{C\alpha_{j}}(W)}\dim(H^\perp_{j} \cap W') -\beta.
\end{equation}
On the other hand, the assumption of global criticality guarantees
\begin{equation}\label{carac-perp2.2}
\sum_{j \in J}  q_{j} \frac{\dim H^\perp_{j}}{\dim H} =\sum_{j \in J}  q_{j}-1
\end{equation}
and the desired inequality follows by plugging \eqref{carac-perp2.2} into \eqref{carac-perp2.1} and dividing by $\dim W$.

The proof of the necessary condition is similar, using this time the second bound from  \Cref{eq:comp sR}.
\end{proof}

\section{Upper bounds for localised regularised  data}
In this section we prove \Cref{pr:rlub}.

\subsection{Lieb's theorem}
A key ingredient is Lieb’s theorem \cite[Theorem 6.2]{Lieb}, which constitutes a central result in the theory of  Brascamp–Lieb inequalities.
It states that there is no loss in restricting the inputs $(f_j)_{j \in J}$ in the definition of the Brascamp-Lieb constant to Gaussian inputs $f_j = \chi_{A_j}$, with positive definite symmetric $A_j \in \End(H_j)$.

We are going to use the following generalisation, which concerns regularised and localised  data. It is extracted from \cite[Corollary 8.15]{BCCT08}.
\begin{thm}[Generalised Lieb's theorem {\cite{BCCT08}}] \label{general-Liebthm}
Let $(\sD,\bbR, T)$ be a localised regularised   Brascamp-Lieb datum with $\sD=((\ell_{j})_{j\in J}, (q_{j})_{j\in J})$ and $\bbR=(R_{j})_{j\in J}$.
Then
\begin{equation}\label{gaussian-upp}
\BL(\sD, \bbR, T)
= \sup_{0<A_{j}\leq R_{j}} \left(\frac{ \prod_{j}(\det A_{j})^{q_{j}}}{\det(T+\sum_{j} q_{j} \ell_{j}^*A_{j}\ell_{j} )}\right)^{1/2},
\end{equation}
where $0 < A_j \leq R_j$ means $A_j$ is ranging through all positive definite symmetric endomorphisms of $H_j$ such that $R_j - A_j$ is positive semi-definite, for each $j \in J$.
\end{thm}

Though not directly related to our task, let us mention there is an effective version of Lieb's theorem, due to Bennett-Bez-Buschenhenke-Cowling-Flock~\cite[Theorem 1.3]{BBBCF}, which estimates how close a regularized Gaussian Brascamp-Lieb constant is  to the Brascamp-Lieb constant. This plays a role  in showing nonlinear Brascamp-Lieb inequalities, another landmark result in the area.


\subsection{Proof of \Cref{pr:rlub}}
The  proof of \Cref{pr:rlub} is inspired by \cite[Section 5]{BCCT08}. The strategy is to  use the generalised Lieb theorem to restrict the estimate of $\BL(\sD, \bbR, T)$ to Gaussian inputs, and then perform explicit computations to bound the latter.

Given $a, b\in \Z$ with $a\leq b$, we use the notation $\llrr{a,b}$ to denote the set of  integers between $a$ and $b$, that is, $\llrr{a,b}=\Z\cap [a,b]$.

\begin{proof}[Proof of \Cref{pr:rlub}]

We first reduce to the case where $\alpha_j = 1$ for all $j \in J$. Indeed, note that for each $j \in J$ and every $W \in \Gr(H)$, we have
$\rk_1( \alpha_j^{-1} \ell_j \mid W) = \rk_{\alpha_j}(\ell_j \mid W)$. Therefore the datum $\sD':=\bigl((\alpha_j^{-1} \ell_j)_j, (q_j)_j\bigr)$ is $(1^{\otimes J}, \beta)$-perceptive and satisfies $\rk_1( \alpha_j^{-1} \ell_j)=\dim H_{j}$ for all $j$.  By a simple change of variables, we also have
$$\BL(\sD, \bR, T)=\prod_{j\in J} \alpha_{j}^{-q_{j}\dim H_{j}} \BL(\sD', \bR', T)$$
 where $\bR'=(\alpha_{j}^2 R_{j})_{j\in J}$, and $N(\sD, \bR, T)=N(\sD', \bR', T)$.
It  is thus sufficient to prove the theorem for $\sD=\sD'$, whence the reduction to $\bba=1^{\otimes J}$.

Now, in view of Lieb's theorem \ref{general-Liebthm}, it suffices to show
\begin{equation}
\label{eq:rlub1}
\frac{\prod_{j\in J} (\det A_{j})^{q_{j}} }{\det(T + \sum_{j \in J} q_j \ell_j^* A_{j}\ell_j)}\leq d^{\sA(\sD)}  \sE(\sD)^2 N(\sD, \bbR, T)^{\sA(\sD)-d+\beta} \norm{T^{-1}}^{\beta}
\end{equation}
for all positive definite symmetric endomorphisms $A_j \in \End(H_j)$ satisfying $A_j \leq R_j$.

Let $M := T+ \sum_{j\in J} q_{j} \,\ell^*_{j} A_{j}\ell_{j} \in \End(H)$.
It is symmetric and positive definite.
Let $(e_1, \dotsc, e_d)$ be an orthonormal basis of $H$ in which the matrix of $M$ is diagonal, with diagonal entries $\lambda_{1}\geq \dots \geq \lambda_{d}>0$.
For a set of indices $I \subset \llrr{1, d}$, we write
\[
V_I = \Span \set{ e_i : i \in I}.
\]
For $k \in \llrr{0, d-1}$, we abbreviate
\(
V_{> k} = V_{\llrr{k+1,d}}.
\)

For each $j \in J$, we construct an index subset $I_j \subset \llrr{1,d}$ satisfying
\begin{equation}
\label{eq:Ij}
\forall i \in I_{j}, \quad \dist(\ell_j e_i, \ell_j V_{I_j \cap (i,d] }) \geq \frac{1}{\sqrt{d}},
\end{equation}
\begin{equation}\label{eq:Ij2nd}
\forall i \in \llrr{1,d},\quad \dist(\ell_j e_i, \ell_j V_{I_j \cap [i,d]}) < \frac{1}{\sqrt{d}}.
\end{equation}
This construction is done inductively.
First, we put $d$ in $I_j$ if and only if $\norm{\ell_j e_d} \geq \frac{1}{\sqrt{d}}$.
This determines $I_j \cap (d-1,d]$.
Then for each integer $i < d$, starting from $d - 1$, suppose the intersection $I_j \cap (i, d]$ is determined.
We put $i$ in $I_j$ if \eqref{eq:Ij} holds, and we skip $i$ and proceed to $i - 1$ otherwise.
This procedure terminates when $i$ reaches $0$, and yields a set $I_{j}$ satisfying \eqref{eq:Ij}, \eqref{eq:Ij2nd}.

We derive further properties of the sets $I_{j}$. First, we claim that for all $k\in \llrr{0, d - 1}$, we have
\begin{equation}\label{Ij-percep-estimate}
\sum_{j\in J} q_{j} \abs{I_j \cap (k, d] }  \geq d - k - \beta.
\end{equation}
To see that, note that by \eqref{eq:Ij2nd}, we have
\begin{equation*}
\ell_j B_{1}^{V_{> k}} \subset \ell_j V_{I_j \cap (k,d]} + B^H_{1}.
\end{equation*}
In view of \Cref{lm:arank}, this implies
\begin{equation}\label{ellVk}
\rk_1(\ell_j \mid V_{> k}) \leq \dim V_{I_j \cap (k, d]} = \abs{I_j \cap (k, d]}.
\end{equation}
Summing over $j \in J$ and using the assumption of perceptivity, the claim \eqref{Ij-percep-estimate} follows.


We also see that $|I_{j}|=\dim H_{j}$. Indeed, taking $k=0$ in \eqref{ellVk} and recalling $\rk_{1}(\ell_j) = \dim H_j$ by hypothesis, we obtain $\dim H_j \leq \abs{I_j}$.
By \eqref{eq:Ij}, we have
\begin{equation}\label{eq:detellj}
\norm{\wedge_{i\in I_{j}}\ell_{j}e_{i} } = \prod_{i \in I_j} \dist\bigl(\ell_j e_i, \ell_j V_{I_j \cap (i,d]}\bigr) \geq d^{-\frac{\abs{I_j}}{2}},
\end{equation}
and in particular $\abs{I_j} \leq \dim H_j$.
Therefore $\abs{I_j}=\dim H_{j}$. Note the argument shows that $(\ell_{j}e_{i})_{i\in I_{j}}$ is a basis of $H_{j}$. By definition of $\sA(\sD)$, we also have
\begin{equation}\label{Ij-cardinal}
\sum_{j}q_{j} \abs{I_{j}} = \sA(\sD).
\end{equation}

We deduce upper bounds on the determinants of the $A_{j}$.
Given $j \in J$, and $i\in I_j$, we have
\begin{equation}\label{bnd-scalar-prod}
{\langle A_{j}\ell_{j}e_{i},  \ell_{j} e_{i}\rangle} = {\langle \ell_{j}^*A_{j}\ell_{j}e_{i}, e_{i}\rangle} \leq \frac{1}{q_j} {\langle M e_{i}, e_{i}\rangle} \leq \frac{\lambda_i}{q_j}.
\end{equation}
Employing \Cref{eq:Choleski} below, then \eqref{eq:detellj} and \eqref{bnd-scalar-prod}, we deduce
\begin{align*}
{\det A_{j}} & \leq \norm{\wedge_{i\in I_{j}}\ell_{j}e_{i} }^{-2} \prod_{i\in I_{j}}{\langle A_{j}\ell_{j}e_{i},  \ell_{j} e_{i}\rangle}\\
&\leq \Bigl(\frac{d}{q_j}\Bigr)^{\abs{I_j}} \prod_{i\in I_{j}}\lambda_{i}.
\end{align*}

Taking power $q_{j}$, then the product over $j$, and using \eqref{Ij-cardinal}, we get
\begin{equation} \label{eqproddetAj}
\prod_{j\in J} (\det A_{j})^{q_{j}} \leq
d^{\sA(\sD)}\sE(\sD)^2
\prod_{i=1}^d{\lambda_{i}}^{\sum_{j \in J} q_{j}\abs{I_{j}\cap \{i\}}}.
\end{equation}
We now bound the product on the eigenvalues $\lambda_{i}$.
Telescoping, then using the properties  \eqref{Ij-percep-estimate} and \eqref{Ij-cardinal} of $I_{j}$, then telescoping again, we have
\begin{align}
 \prod_{i=1}^d{\lambda_{i}}^{\sum_{j}q_{j}\abs{I_{j} \cap \{i\}}}
 &= \lambda_{1}^{\sum_{j}q_{j}\abs{I_j}}\prod_{k=1}^{d-1} \left(\frac{\lambda_{k+1}}{\lambda_k}\right)^{\sum_{j}q_{j}\abs{I_{j}\cap (k,d]} } \nonumber\\
 &\leq \lambda_{1}^{\sA(\sD)}\prod_{k=1}^{d-1} \left(\frac{\lambda_{k+1}}{\lambda_k}\right)^{d-k-\beta} \nonumber\\
 &=\lambda_{1}^{\sA(\sD)-d+1+\beta} \lambda_{2}\lambda_{3}\dotsm \lambda_{d-1} \lambda_{d}^{1-\beta} \nonumber\\
 &= \lambda_{1}^{\sA(\sD)-d+\beta} \lambda_{d}^{-\beta}  \det M. \label{eqprodlambdai}
 \end{align}
Equations \eqref{eqproddetAj} and \eqref{eqprodlambdai} together yield
$$
\frac{\prod_{j\in J} (\det A_{j})^{q_{j}} }{{\det M}}\leq d^{\sA(\sD)}\sE(\sD)^2
\lambda_{1}^{\sA(\sD)-d+\beta} \lambda_{d}^{-\beta}.
$$

Now since the Gaussian input $(A_j)_j$ is dominated by $(R_j)_j$
we have $\lambda_{1}\leq N(\sD, \bbR, T)$.
On the other hand, $\lambda_{d} \geq \normbig{T^{-1}}^{-1}$.
This concludes the proof of \eqref{eq:rlub1}.
\end{proof}

In the above proof, we used the following elementary lemma.
\begin{lemma}\label{eq:Choleski}
Let $H$ be a Euclidean space, let $Q\in \End(H)$ be a positive semi-definite symmetric endomorphism,  let  $(\eps_{1}, \dots, \eps_{d})$ be a basis of $H$.
Then
\begin{equation*}
\det Q \leq \|\eps_{1}\wedge \dots \wedge \eps_{d}\|^{-2}\prod_{k=1}^d \langle Q \eps_{k}, \eps_{k}\rangle.
\end{equation*}
\end{lemma}

\begin{proof}
If  $(\eps_{1}, \dots, \eps_{d})$ is an orthonormal basis, then $\|\eps_{1}\wedge \dots \wedge \eps_{d}\|=1$ and the result follows by  the Choleski decomposition, which states  that $Q=F^*F$ where $F$ is upper triangular in the basis $(\eps_{1}, \dots, \eps_{d})$. For the general case, we  write $(\eps_{1}, \dots, \eps_{d})=A(\eps'_{1}, \dots, \eps'_{d})$ where $A\in \GL_{d}(\R)$ and the $\eps'_{i}$ form an orthonormal basis. We then have
$$\det Q = (\det A)^{-2}\det A^*Q A \leq \|\eps_{1}\wedge \dots \wedge \eps_{d}\|^{-2} \prod_{k=1}^d \langle A^*Q A \eps'_{k}, \eps'_{k}\rangle,$$
hence the result.
\end{proof}

\section{Lower bound for localised regularised  data}

In this section we show \Cref{thm:loc-lb}, which states a lower bound for localised regularised   Brascamp-Lieb constants.

\begin{proof}[Proof of \Cref{thm:loc-lb}]
Consider a collection  $(A_j)_{j \in J}$ of positive definite symmetric endomorphisms $A_{j}\in \End(H_{j})$ with $A_{j}\leq  \id_{H_j}$. The localised regularised  Brascamp-Lieb inequality with  test functions $f_{j}=\chi_{A_{j}}$ (definition in \eqref{def-chi-sN}) yields by direct computation
\begin{equation}\label{BLDRT-lwb}
\BL(\sD, \bR, T) \geq \left(\frac{\prod_{j \in J} (\det A_j)^{q_{j}}}{\det\left(T + \sum\nolimits_{j \in J} q_j \ell_j^* A_j \ell_j \right)}\right)^{1/2}.
\end{equation}
The strategy is to choose the $A_{j}$ so that the right-hand side of \eqref{BLDRT-lwb} has the desired lower bound.

By \Cref{lm:arank}, for each $j \in J$, there exists $E_j \in \Gr(H_j)$ of dimension $\dim E_j = \rk_\alpha(\ell_j \mid W)$ such that
\[
\ell_j B^W_1 \subset B^{H_j}_\alpha + E_j.
\]
Let $A_j \in \End(H_j)$ be the positive definite symmetric endomorphism whose square root is
\[
A_j^{1/2} = \alpha \id_{E_j} \oplus \id_{E_j^\perp}.
\]
Observe that $A_j \leq \id_{H_j}$ and $\det(A_j)^{1/2} = \alpha^{\rk_\alpha(\ell_j \mid W)}$. It remains to bound from above the determinant of $T+\sum_{j \in J} q_j \ell_j^* A_j \ell_j$. We assume from the start that $E_{j}$ has the property described in the remark after \Cref{lm:arank}.
In particular, for any unit vector $w \in W$, we can decompose $w = w' + w''$ into orthogonal vectors satisfying $\ell_j w' \in E_j$, $\ell_j w'' \in E_j^\perp$, $\norm{\ell_j w'} \leq C\norm{w'}$ and $\norm{\ell_j w''} \leq \alpha \norm{w''}$.
It follows that
\begin{align*}
\langle \ell_j^* A_j \ell_j w, w \rangle &= \norm{A_j^{1/2} \ell_j w}^2 \\
&= \norm{ \alpha\ell_j w' + \ell_j w''}^2\\
&= \alpha^2 \norm{\ell_j w'}^2 + \norm{\ell_j w''}^2\\
&\leq \alpha^2 C^2\norm{w'}^2 + \alpha^2 \norm{w''}^2\\
&\leq \alpha^2 C^2
\end{align*}
On the other hand, for every unit vector $w_{\perp} \in W^\perp$, we have
\[
\langle \ell_j^* A_j \ell_j w_{\perp}, w_{\perp} \rangle = \norm{A_j^{1/2} \ell_j w_{\perp}}^2 \leq C^2.
\]
Consider now an orthonormal basis $(e_1, \dotsc, e_d)$ of $H$, obtain by joining an orthonormal basis of $W$ with one of $W^\perp$.
Using \Cref{eq:Choleski}, the above inequalities, then $C\geq 1\geq \alpha$, we find
\begin{align*}
\det\left(T + \sum\nolimits_{j \in J} q_j \ell_j^* A_j \ell_j \right) &\leq \prod_{i = 1}^d \left\langle Te_i + \sum\nolimits_{j \in J} q_j \ell_j^* A_j \ell_j e_i, \,e_i \right\rangle\\
&\leq \Big(t + \sum\nolimits_{j \in J} q_j \alpha^2C^2 \Bigr)^{\dim W} \left(t + \sum\nolimits_{j \in J} q_j C^2\right)^{\dim W^\perp}\\
&\leq  \alpha^{2 \dim W}\left( (C/\alpha)^2 t + C^2\sum\nolimits_{j \in J} q_j \right)^d.
\end{align*}
By combination with \eqref{BLDRT-lwb}, this yields the announced lower bound.
\end{proof}

\section{Bounds  on Brascamp-Lieb constants}  \label{Sec-compute-BL}

We derive the estimates on $\BL(\sD)$  from their counterparts for $\BL(\sD, \bbR, T)$.

\subsection{Brascamp-Lieb constants as limits}
It is well known that $\BL(\sD)$ can be expressed as a limit of localised regularised Brascamp-Lieb constants. We recall this standard fact.

 First, we see how to pass to the limit to remove the truncation by $T$. Let $\sD = ((\ell_j)_{j \in J}, (q_j)_{j \in J})$ be  a Brascamp-Lieb datum, and $\bbR = (R_{j})_{j \in J}$ a collection such that each $R_j$ is a positive definite symmetric endomorphism of $H_j$.
Define the regularised Brascamp-Lieb constant $\BL(\sD, \bbR)$ to be $\BL(\sD, \bbR)=\BL(\sD, \bbR, 0)$. More formally, it is the smallest number $\BL(\sD, \bbR)$ such that
\[
\int_{H} \prod_{j \in J} (f_{j}\circ \ell_{j})^{q_{j}} \leq \BL(\sD, \bbR) \prod_{j \in J} \left(\int_{H_{j}} f_{j} \right)^{q_{j}},
\]
for all inputs $(f_{j})_{j \in J}$ of functions satisfying that $f_j$ is of type $R_j$ for each $j \in J$.

\begin{lemma}\label{lm:BLlimT}
Given a regularised Brascamp-Lieb datum $(\sD, \bbR)$, we have
\[
\BL(\sD, \bbR) = \lim_{T \to 0} \BL(\sD, \bbR, T),
\]
where $T \to 0$ means that $T$ converges to $0$ in $\End(H)$ while staying symmetric and positive definite.
\end{lemma}
\begin{proof}
This is a simple application of Fatou's Lemma. 
\end{proof}

Next, we get rid of the regularisation.
\begin{lemma}[{\cite[Eq. (44)]{BCCT08}}]
\label{lm:BLlimR}
Given a regularised Brascamp-Lieb datum $(\sD, \bbR)$, we have 
\[
\BL(\sD) = \lim_{t \to +\infty} \BL(\sD, t \bbR),
\]
where $t \bbR$ is the coordinate-wise scalar multiplication of $\bbR$ by $t \in \R_{>0}$.
\end{lemma}

\subsection{Proof of Theorems \ref{thm:ub}, \ref{thm:lb}}
We can now justify the upper bound  \Cref{thm:ub} and the lower bound \Cref{thm:lb} on Brascamp-Lieb constants
\begin{proof}[Proof of \Cref{thm:ub} ]
Note that being globally critical and $\bba$-perceptive implies in particular that
\[
\sum_{j \in J} q_j \rk_{\alpha_j}(\ell_j) \geq \dim H = \sum_{j \in J} q_j \dim H_j,
\]
hence $\rk_{\alpha_j}(\ell_j) = \dim H_j$ for every $j \in J$.
This allows us to apply \Cref{pr:rlub} with $\beta=0$ and for any regularisation $\bbR$ and localisation $T$.
As $\sA(\sD) = d$ and $\beta=0$, we obtain
\[
\BL(\sD, \bR, T) \leq d^{\frac{d}{2}} \sE(\sD) \Upsilon(\sD) \prod_{j \in J} \alpha_j^{- q_j \dim H_j}.
\]
Note this bound  is independent of $\bbR$ and $T$.
Applying Lemmas \ref{lm:BLlimT}, \ref{lm:BLlimR}, we see this inequality also holds for $\BL(\sD)$, whence the result.
\end{proof}

\begin{proof}[Proof of \Cref{thm:lb}]
It follows from \Cref{thm:loc-lb},
using the trivial inequalities
\[
\BL(\sD) \geq \BL(\sD, \bbR) \geq \BL(\sD, \bbR, T)
\]
for any regularisation $\bbR$ and localisation $T$.
\end{proof}

\section{Visual inequality} \label{Sec-visual-inequality}

We prove the visual inequality announced in the introduction.

\begin{proof}[Proof of \Cref{thm-visual-ineq}]
We may assume that $A$ is a union of balls of radius $\delta$.
Let $\bR=(R_j)_{j \in J}$ by given by $R_j = \delta^{-2} \Id_{H_j}$, and $T=\Id_{H}$. Set  $f_{j}=\1_{\pi_{H_{j}}(A)} * \sN_{R_{j}}$ (it can be seen as a mollification of the projection $\pi_{H_j}(A)$ at scale $\delta$).
Observe that
$$\1_{A} \ll O_{d}(1)^{\sum_{j}q_{j}} \prod_{j \in J} (f_{j}\circ \pi_{H_j})^{q_{j}} \,\chi_{T}.$$
Integrating over $H$, then using the definition of $\BL(\sD, \bbR, T)$ and  that each $f_{j}$ is of  type $R_{j}$,  we obtain
\begin{equation}\label{eq-viseqproof1}
\cN_{\delta}(A)\delta^d \ll_{d} O_{d}(1)^{\sum_{j}q_{j}}\BL(\sD, \bbR, T) \prod_{j \in J}(\cN_{\delta}(\pi_{H_j}A) \delta^{\dim H_j})^{q_{j}}.
 \end{equation}
 Using \Cref{pr:rlub}, we may further  bound
\begin{equation}\label{eq-viseqproof2}
\BL(\sD, \bbR, T) \leq d^{\frac{\sA(\sD)}{2}} \sE(\sD) \prod_{j \in J} \alpha_j^{- q_j \dim H_j} N(\sD, \bbR, T)^{\frac{\sA(\sD)-d+\beta}{2}}.
\end{equation}
Noting that
\[
N(\sD, \bbR, T) \leq
\Bigl(1+\sum\nolimits_{j \in J}q_{j} \delta^{-2}\Bigr)
\]
and $1 \leq \delta^{-2}$, the  last term in \eqref{eq-viseqproof2} is dominated by
\begin{equation}\label{eq-viseqproof3}
N(\sD, \bbR, T)^{\frac{\sA(\sD)-d+\beta}{2}} \leq
\Bigl(1+\sum\nolimits_{j \in J}q_{j}\Bigr)^{\frac{\sA(\sD)-d+\beta}{2}} \delta^{-\sA(\sD)+d-\beta}.
\end{equation}
Combining \eqref{eq-viseqproof1}, \eqref{eq-viseqproof2}, \eqref{eq-viseqproof3}, and cancelling out the terms $\delta^{-\sA(\sD)+d}$, we obtain the inequality announced in \Cref{thm-visual-ineq}.
\end{proof}

\begin{remark}
It follows from this proof that the exponent $\frac{\sA(\sD)-d+\beta}{2}$ appearing in \Cref{pr:rlub} is optimal. Indeed, if \Cref{pr:rlub} were to hold with a smaller exponent, say $\frac{\sA(\sD)-d+\beta- \eps}{2}$ with $\eps>0$, then the above argument would yield \eqref{eq-thm-visual-ineq} with the term $\delta^{-\beta}$ replaced by $\delta^{-\beta+\eps}$, but such an inequality clearly fails as noticed earlier in \S\ref{Sec-appli}.
\end{remark}

\begin{remark}
Instead of considering the $\delta$-covering number of a bounded subset of $\R^d$, we may zoom out and reduce to the $1$-covering number of a dilated set, that is, $\cN_{\delta}(A)=\cN_{1}(\delta^{-1}A)$.
This renormalisation allows to perform the proof of \Cref{thm-visual-ineq} using this time a fixed resolution $\bbR=(\Id_{H_{j}})_{j}$ and increasing the truncation domains, instead of increasing resolution while fixing the truncation. In this case the necessary extra term $\delta^{-\beta}$ in \Cref{thm-visual-ineq} will come from the term $\|T^{-1}\|^{\frac{\beta}{2}}$ instead of $N(\sD, \bbR, T)^{\frac{\sA(\sD)-d+\beta}{2}}$ in \Cref{pr:rlub}.
As in the above remark, this justifies the optimality of the exponent $\beta/2$ in \Cref{pr:rlub}.
\end{remark}

\appendix
\section{On separating geometry and distortion} \label{Sec-appendix}
We provide an example where the variant \eqref{up-v2} of  \Cref{thm:ub} mentioned in the introduction is more efficient than the original upper bound in \Cref{thm:ub}.
\begin{example}
Let $\lambda \in (0,1]$ be a parameter.
Consider the Brascamp-Lieb datum $\sD_\lambda = \bigl((\ell_1, \ell_2, \ell_3), (\frac{1}{2},\frac{1}{2}, \frac{1}{2})\bigr)$ where the $\ell_j \colon \R^3 \to \R^2$ are given by: $\forall x = (x_1, x_2, x_3) \in \R^3$,
\begin{align*}
&\ell_1 (x) = (x_2, x_3) \quad \quad \ell_2 (x) = (x_1, x_3) \quad \quad \ell_3 (x) = (x_1, \lambda x_2).
\end{align*}
For $\lambda = 1$, $\sD_1$ is the Loomis-Whitney datum from $\R^3$ to $\R^2$.
It is  known that $\BL(\sD_1) = 1$ (see e.g. \cite[Example 1.6]{BCCT08}).
For general $\lambda \in \R_{>0}$, a simple change of variables shows that
\[
\BL(\sD_\lambda) =  \lambda^{-\frac{1}{2}} \BL(\sD_1) = \lambda^{-\frac{1}{2}}.
\]
Note that the configuration of the kernels $\Ker \ell_{j}$ from $\sD_\lambda$ is independent of $\lambda$, being that of a Loomis-Whitney scenario.
The variation of $\BL(\sD_\lambda)$ in  $\lambda$ is only due to the distortion of $\ell_{3}$.

We can apply the variant   \eqref{up-v2}  to $\sD_\lambda$.
The distortion term $\Upsilon(\sD_\lambda)$ is precisely $\lambda^{-\frac{1}{2}}$.
Other terms in the upper bound \eqref{up-v2}  are independent of $\lambda$. In fact $(\sD_\lambda)_{\mathrm{proj}}$ is
$(\alpha_1, \alpha_2, \alpha_3)$-perceptive if and only if $\alpha_1^2 + \alpha_2^2 + \alpha_3^2 < 1$.
The  variant \eqref{up-v2} then gives $\BL(\sD_\lambda) \leq 3^3 2^{\frac{3}{2}}\lambda^{-\frac{1}{2}}$. Although not sharp, it captures the correct asymptotic in $\lambda$.

On the contrary, if we use \Cref{thm:ub} directly instead, we obtain the wrong asymptotic.
Indeed, testing the requirement \eqref{def-alpha-percep} with $W = \R(0,1,0)$, we see that for  $\sD_{\lambda}$ to be $(\alpha_1,\alpha_2,\alpha_3)$-perceptive,  we need at least $\alpha_3 < \lambda$.
This results in the upper bound of \Cref{thm:ub} having the term  $\alpha_3^{-q_3 \dim H_3} > \lambda^{-1}$. Hence the claim.

In conclusion,  the  upper bound variant \eqref{up-v2}  for Brascamp-Lieb constants can be more relevant than the original  \Cref{thm:ub} when dealing with linear maps which are far from being projectors.

\bigskip

\end{example}

\bibliographystyle{plain}
\bibliography{effectiveBL}
\end{document}